\documentclass{amsart}
\usepackage[all]{xy}
\usepackage{amsmath}
\usepackage{amssymb}
\usepackage{amsthm}
\usepackage{amscd}

\usepackage{mathtools}
\usepackage{hyperref}
\usepackage{color}
\definecolor{midpurple}{rgb}{0.6,0.2,0.4}
\hypersetup{colorlinks=true,citecolor=blue,linkcolor=blue,urlcolor=midpurple}
\numberwithin{equation}{section}

\newcommand{\R}{\mathcal{R}}

\newcommand{\V}{\mathcal{V}}

\newcommand{\F}{\mathcal{F}}

\newcommand{\PP}{\Bbb P}

\newcommand{\Z}{\Bbb Z}

\newcommand{\RR}{\Bbb R}
\newcommand{\TT}{\mathbb{ T} }

\newcommand{\QQ}{\Bbb Q}
\newcommand{\C}{\Bbb C}

\newcommand{\Ca}{\mathcal{C}}

\newtheorem{theorem}{Theorem}
\newtheorem{definition}[theorem]{Definition}
\newtheorem{proposition}[theorem]{Proposition}

\newtheorem{lemma}[theorem]{Lemma}

\newtheorem{remark}[theorem]{Remark}

\DeclareMathOperator{\sgn}{sgn}

\DeclareMathOperator{\supp}{supp}

\DeclareMathOperator{\diam}{diam}

\DeclareMathOperator{\N}{\mathbb{N}}

\numberwithin{theorem}{section}

\begin{document}
\title[Discrete Carleson Along the Primes]{A Discrete Carleson Theorem Along the Primes with a Restricted Supremum}
%Enter your title between curly braces
\author{Laura Cladek}
\address{
Department of Mathematics
The University of Wisconsin-Madison \\
480 Lincoln Dr.
Madison, WI 53706-1325}
\email{cladek@math.wisc.edu}
\author{Kevin Henriot}
\address{
Department of Mathematics
The University of British Columbia \\
1984 Mathematics Road
Vancouver, B.C.
Canada V6T 1Z2}
\email{khenriot@math.ubc.ca}
\author{Ben Krause}
\address{
Department of Mathematics
The University of British Columbia \\
1984 Mathematics Road
Vancouver, B.C.
Canada V6T 1Z2}
\email{benkrause@math.ubc.ca}
\author{Izabella $\L$aba}
\address{
Department of Mathematics
The University of British Columbia \\
1984 Mathematics Road
Vancouver, B.C.
Canada V6T 1Z2}
\email{ilaba@math.ubc.ca}
\author{Malabika Pramanik}
\address{
Department of Mathematics
The University of British Columbia \\
1984 Mathematics Road
Vancouver, B.C.
Canada V6T 1Z2}
\email{malabika@math.ubc.ca}

\date{\today}
\maketitle

\begin{abstract}
Consider the discrete maximal function acting on finitely supported functions on the integers,
\[ \Ca_\Lambda f(n) := \sup_{\lambda \in \Lambda} \left| \sum_{p \in \pm \PP} f(n-p) \log |p| \frac{e^{2\pi i \lambda p}}{p} \right|,\]
where $\pm \PP := \{ \pm p : p \text{ is a prime} \}$, and $\Lambda \subset [0,1]$. We give sufficient conditions on $\Lambda$, met by (finite unions of) lacunary sets, for this to be a bounded sublinear operator on $\ell^p(\Z)$ for $\frac{3}{2} < p < 4$. This result has as its precedent the work of Bourgain and Wierdl on the discrete maximal function along the primes,
\[ \sup_{N} \left| \frac{1}{N} \sum_{p \in \PP, p \leq N} f(n-p) \log p \right|,\]
and Mirek and Trojan on maximally truncated singular integrals along the primes,
\[ \sup_{N} \left| \sum_{p \in \pm \PP, |p| \leq N} f(n-p) \frac{\log |p|}{p} \right|,\]
as well as the most recent work of Mirek, Trojan, and Zorin-Kranich, studying the larger variational variants of the above maximal operators.

The proof relies on Bourgain's elegant ``multi-frequency logarithmic lemma", developed in his study of pointwise ergodic theorems along polynomial orbits, combined with the results of Hyt\"{o}nen and Lacey on a vector-valued version of Carleson's theorem, and those of Oberlin, Seeger, Tao, Thiele, and Wright on a variational version of Carleson's theorem.
\end{abstract}

\section{Introduction}
Modern discrete harmonic analysis began in the late 80s/early 90s, when Bourgain proved his celebrated pointwise ergodic theorems along polynomial orbits \cite{B0, B, B1}; a key ingredient in the proof of Bourgain's polynomial ergodic theorems were the quantitative estimates
\[ \| \sup_N \left| \frac{1}{N} \sum_{n\leq N } f(x - P(n)) \right| \|_{\ell^p(\Z)} \leq C_{p,P} \| f\|_{\ell^p(\Z)}, \ P \in \Z[-], \ 1<p \leq \infty\]
for some absolute constants $C_{p,P} >0$.

The analogous result for averages along the primes was established by Bourgain \cite{B'} and Wierdl in \cite{W}:
\[ \| \sup_N \left| \frac{1}{N} \sum_{p \in \PP, p \leq N} f(x- p) \log p \right|  \|_{\ell^p(\Z)} \leq C_p \| f \|_{\ell^p(\Z)}, \ 1< p \leq \infty\]
for some absolute constant $C_p > 0$. Here, $\PP := \{ p : p \text{ is a prime} \}$.

Since Bourgain's and Wierdl's work on ``arithmetic" ergodic theorems, much work has been devoted to studying discrete maximal functions and discrete singular integrals. For some foundational papers in this direction, we refer the reader to \cite{IW, MSW, SW1, SW2} and to the references contained therein.

In 2013, Mirek and Trojan in \cite{MT} began a study of maximally truncated discrete ``arithmetic" singular integral operators, by proving an analogue of Cotlar's ergodic theorem along the primes. The key quantitative result of their paper is the following:

\begin{theorem}[Theorem 2 of \cite{MT}]\label{cot}
The maximal function
\[ \sup_N \left| \sum_{p \in \pm \PP, |p| \leq N} f(x-p) \frac{\log|p|}{p} \right| \]
is $\ell^p(\Z)$ bounded for each $1<p<\infty$.
\end{theorem}

After Mirek and Trojan's result came to light, further study of averaging operators along the primes was conducted by Zorin-Kranich in \cite{zk}, where a variational extension of Bourgain-Wierdl was proven. This result was later recovered by Mirek, Trojan, and Zorin-Kranich in \cite{MTZ}, where it was presented with the variational strengthening of Theorem \ref{cot}; we refer the reader to $\S 2.4$ below for a further discussion of variation.

Most recently, in December of 2015, two seminal papers by Mirek, Stein, and Trojan \cite{MST1, MST2} established joint extensions of Bourgain \cite{B1} and Mirek, Trojan, and Zorin-Kranich \cite{MTZ}, proving -- among other things -- maximal and variational estimates for averaging operators and truncations of singular integrals along polynomial sequences. In light of their work, it is natural to begin considering discrete analogues of operators which enjoy an additional modulation symmetry -- \emph{Carleson} operators. This work was initiated by Eli Stein, who we have learned \cite{SS} was able to bound the discrete Carleson operator on $\ell^p(\Z)$, $1 < p < \infty$:

\begin{theorem} The discrete Carleson operator
\[ \sup_{0 \leq \lambda \leq 1} \left| \sum_{m \neq 0} f(n-m) \frac{e(\lambda m)}{m} \right| \]
is bounded on $\ell^p(\Z), \ 1<p<\infty$. Here and throughout, $e(t) := e^{2\pi i t}$.
\end{theorem}
Stein was able to \emph{transfer} this result from its celebrated continuous counterpart;  the fact that the phases $m \mapsto \lambda m$ are \emph{linear} plays a crucial role in this transference. Without this linearity, Krause and Lacey in \cite{KL} were only partially able to develop an $\ell^2(\Z)$-theory for the discrete quadratic Carleson operator
\[\sup_{0 \leq \lambda \leq 1} \left| \sum_{m \neq 0} f(n-m) \frac{e(\lambda m^2)}{m} \right|.  \]
Motivated by this work, Mirek proposed the study of the discrete Carleson operator along the primes
\[\Ca_{\PP} f(n):= \sup_{0 \leq \lambda \leq 1} \left| \sum_{p \in \pm \PP } f(n-p) \log |p| \frac{e(\lambda p) }{p} \right|;  \]
in this paper we begin a study of $\Ca_{\PP}$ by exhibiting a class of infinite sets $\Lambda \subset [0,1]$ for which the restricted maximal function
\[\Ca_{\Lambda} f(n):= \sup_{\lambda \in \Lambda} \left| \sum_{p \in \pm \PP } f(n-p) \log |p| \frac{e(\lambda p) }{p} \right| \]
is $\ell^p(\Z)$ bounded, in the range $\frac{3}{2} < p < 4$.

The type of condition that we impose on $\Lambda$ is described here.

For a subset $\Lambda \subset [0,1]$, define $\mathcal{N}(\delta) = \mathcal{N}_\Lambda(\delta)$ to be the smallest number of intervals of length $\delta$ needed to cover $\Lambda$, and set
\[ C_d := C_{d,\Lambda} := \sup_{0< \delta < 1} \delta^d \mathcal{N}(\delta).\]

\begin{definition} A subset $\Lambda \subset [0,1]$ is called a ``pseudo-lacunary set" if for each $j \geq 1$,
\[ C_{1/j} \leq A j^{M},\]
for some $A,M > 0 $ absolute constants.
\end{definition}
We note that pseudo-lacunary sets are zero-Minkowski-dimensional. As the name suggests, any lacunary set is automatically pseudo-lacunary: if $\Lambda = \{ \rho^{-k} : k \geq 0 \}$, where $\rho > 1$, then
\[ \mathcal{N}(\delta) = k+2 \]
for
\[ \rho^{-k-1} \leq \delta < \rho^{-k},\]
and consequently,
\[ C_{1/j} \leq \sup_{k\geq 0} \ (k+2) \rho^{-\frac{k}{j}} \leq A_\rho \cdot j,\]
where $A_\rho > 0$ is an absolute constant.
 
Under the above condition, we prove the following theorem.
\begin{theorem}\label{Main} Suppose $\frac{3}{2} < p < 4$. Then for any pseudo-lacunary set of modulation parameters, $\Lambda \subset [0,1]$, the maximal function
\[ \Ca_\Lambda f(n) := \sup_{\lambda \in \Lambda} \left| \sum_{p \in \pm \PP} f(n-p) \log |p| \frac{e(\lambda p)}{p} \right|,\]
is bounded on $\ell^p(\Z)$.
\end{theorem}

The structure of the proof is as follows. Following the lead of Krause and Lacey \cite{KL}, we treat the inequality above as a maximal multiplier theorem, where the multipliers are a function of $\lambda$, given by
\[ m(\lambda,\beta) = \sum_{p \in \pm \PP} \log|p| \frac{e(\lambda p - \beta p)}{p}.\]
Using the work of \cite{MT}, a detailed description of $m(\lambda,\beta)$ is available. The analysis splits into the usual major/minor arc dichotomy.

On the minor arcs, little is known about the multiplier $m(\lambda,\beta)$ except that it is ``small;" a major difficulty is that the minor arcs vary with $\lambda$. As in \cite{KL}, we use trivial derivative estimates and the technical Lemma \ref{Mink} to get around this difficulty. This is one of the places where the pseudo-lacunarity hypothesis is needed.

On each individual major arc, centered at rationals with small denominators, the multiplier looks like a weighted version of the Carleson operator. Taking all rationals with small denominators into account leads to a ``multi-frequency Carleson-type operator" -- the Carleson analogue of the multi-frequency averaging operator considered by Bourgain in his work on polynomial ergodic theorems, \cite{B1}. Fortunately, after another appeal to pseudo-lacunarity, Bourgain's method becomes robust enough to apply to our setting; the key analytic ingredient required is the variational Carleson theorem, due to Oberlin, Seeger, Tao, Thiele, and Wright \cite[Theorem 1.2]{OSTTW}.

The structure of the paper is as follows:

In $\S 2$ we collect several preliminary tools;

In $\S 3$ we adapt Bourgain's argument to our setting;

In $\S 4$ we reduce the study of our operator to one amenable to our multi-frequency Carleson theory; and

In $\S 5$ we complete the proof.

\subsection{Acknowledgements} The authors wish to thank Michael Lacey, first and foremost, for his input and support. They would also like to thank Francesco di Plinio for many helpful conversations on the subject of time frequency analysis, and to Gennady Uraltsev for sharing with us Lemma \ref{gen}. Finally, they would like to thank Mariusz Mirek for proposing the problem, and for his early encouragement. 
This research was supported by NSERC Discovery Grants
22R80520 and 22R82900, and by NSF Research and Training Grant DMS 1147523. A substantial part of this work was done at the Institute for Computational and Experimental Research in Mathematics (ICERM), and the authors would like to thank ICERM for its hospitality.

\section{Preliminaries}
\subsection{Notation}
As previously mentioned, we let $e(t) := e^{2\pi i t}$. Throughout, $0<c<1$ will be a small number whose precise value may differ from line to line.

For finitely supported functions on the integers, we define the Fourier transform
\[ \F f(\beta) := \hat{f}(\beta) := \sum_n f(n) e(-\beta n),\]
with inverse
\[ \F^{-1}g(n) := g^{\vee}(n) := \int_{\TT} g(\beta) e(\beta n) \ d\beta.\]
For Schwartz functions on the line, we define the Fourier transform
\[ \F f(\xi) := \hat{f}(\xi) := \int f(x) e(-\xi x) \ dx,\]
with inverse
\[ \F^{-1}g(x) := g^{\vee}(x) := \int g(\xi) e(\xi x) \ d\xi.\]

We will also fix a smooth dyadic resolution of the function $\frac{1}{t}$. Thus,
\[ \frac{1}{t} = \sum_{j \in \Z} \psi_j(t) := \sum_{j \in \Z} 2^{-j} \psi(2^{-j}t), \; \; \; t \neq 0,\]
where $\psi$ is a smooth odd function satisfying $|\psi(x)| \lesssim 1_{[1/4,1]}(|x|)$. This can be accomplished by setting
\[ \psi(x) = \frac{ \eta(x) - \eta(2x) }{x},\]
where $1_{[-1/2,1/2]} \leq \eta \leq 1_{[-1,1]}$ is a Schwartz function. We will mostly be concerned with the regime $|t| \geq 2$, so we will often restrict our attention to
\[ \sum_{j\geq 2} \psi_j(t),\]
which agrees with $\frac{1}{t}$ for $|t| \geq 2$.

Finally, since we will be concerned with establishing a priori $\ell^p(\Z)$- or $L^p(\RR)$- estimates in this paper, we will restrict every function considered to be a member of a ``nice" dense subclass: each function on the integers will be assumed to have \emph{finite support}, and each function on the line will be assumed to be a \emph{Schwartz function.}

We will make use of the modified Vinogradov notation. We use $X \lesssim Y$, or $Y \gtrsim X$ to denote the estimate $X \leq CY$ for an absolute constant $C$. We use $X \approx Y$ as shorthand for $Y \lesssim X \lesssim Y$.

\subsection{Transference}
Although the focus of this paper is proving a discrete inequality, it will be convenient to conduct much of our analysis on the line, where we can take advantage of the dilation structure. The key tool that allows us to ``lift" our real-variable theory to the integers is the following lemma due to Magyar, Stein, and Wainger \cite[Lemma 2.1]{MSW}.

\begin{lemma}[Special Case]\label{trans}
Let $B_1,B_2$ be finite-dimensional Banach spaces, and 
\[ m: \RR \to L(B_1,B_2) \]
be a bounded function supported on a cube with side length one containing the origin that acts as a Fourier multiplier from
\[ L^p(\RR,B_1) \to L^p(\RR, B_2),\]
for some $1 \leq p \leq \infty.$ Here, $L^p(\RR,B):= \{ f: \RR \to B : \| \| f\|_B \|_{L^p(\RR)} < \infty\}$.
Define
\[ m_{\text{per}}(\beta) := \sum_{l \in \Z} m(\beta - l) \ \text{ for } \beta \in \TT.\]
Then the multiplier operator
\[ \| m_{\text{per}} \|_{\ell^p(\Z,B_1) \to \ell^p(\Z,B_2)} \lesssim \| m \|_{L^p(\RR,B_1) \to L^p(\RR,B_2)}.\]
The implied constant is independent of $p, B_1,$ and $B_2$.
\end{lemma}

We will often be in the position of estimating certain Fourier multipliers. The sharpest tool in this regard is the following lemma due to Krause \cite[Theorem 1.5]{K}. We begin by recalling the $r$-variation of a function, $0< r < \infty$:
\[ \| m\|_{\V^r(\RR)} := \sup \left( \sum_{i} |m(\xi_i) - m(\xi_{i+1})|^r \right)^{1/r},\]
where the supremum runs over all finite increasing subsequences $\{\xi_i\} \subset \RR$. By the nesting of little $\ell^p$ spaces, it is easy to see that for $r \geq 1$, the $r$-variation is no larger than the total variation
\[ \| m\|_{\V^1(\RR)} := \sup \left( \sum_{i} |m(\xi_i) - m(\xi_{i+1})| \right),\]
with the same restrictions on the supremum, which agrees with $\| m' \|_{L^1(\RR)}$ for differentiable functions. Then:

\begin{lemma}[Special Case]\label{K1}
Suppose $\{m_\omega: \omega \in X\}$ are a finite collection of (say) $C^1$ functions compactly supported on disjoint intervals $\omega \subset \RR$, and $\{c_\omega : \omega \in X \} \subset \C$ are a collection of scalars. 
Then
\[ \| \left( \sum_{\omega \in X} c_\omega m_\omega \hat{f} \right)^{\vee} \|_{L^p(\RR)} \lesssim \sup_{\omega \in X} |c_\omega| \cdot |X|^{|1/2 - 1/p|} \cdot \sup_{\omega \in X} \| m_\omega \|_{\V^2(\RR)} \cdot \| f\|_{L^p(\RR)}\]
for each $1 < p < \infty$.

In particular, if $\| (m_\omega)' \|_{L^1(\RR)} \lesssim 1$ for each $\omega \in X$,
\[ \| \left( \sum_{\omega \in X} c_\omega m_\omega \hat{f} \right)^{\vee} \|_{L^p(\RR)} \lesssim \sup_{\omega \in X} |c_\omega| \cdot |X|^{|1/2 - 1/p|} \cdot \| f\|_{L^p(\RR)}\]
for each $1 < p < \infty$.
\end{lemma}

In the special case where each $\omega \subset A \cong \TT$, a fundamental domain for the torus with $A \subset [-1,1]$, we may apply Lemma \ref{trans} with $B_1 = B_2 = \C$.
\begin{lemma}\label{KK}
Suppose $\{m_\omega: \omega \in X\}$ are a finite collection of (say) $C^1$ functions compactly supported on disjoint intervals $\omega \subset A \cong \TT$, $A \subset [-1,1]$, and $\{c_\omega : \omega \in X \} \subset \C$ are a collection of scalars. 
Then
\[ \| \left( \sum_{\omega \in X} c_\omega m_\omega \hat{f} \right)^{\vee} \|_{\ell^p(\Z)} \lesssim \sup_{\omega \in X} |c_\omega| \cdot |X|^{|1/2 - 1/p|} \cdot \sup_{\omega \in X} \| m_\omega \|_{\V^2(A)} \cdot \| f\|_{\ell^p(\Z)}\]
for each $1 < p < \infty$.

In particular, if $\| (m_\omega)' \|_{L^1(A)} \lesssim 1$ for each $\omega \in X$,
\[ \| \left( \sum_{\omega \in X} c_\omega m_\omega \hat{f} \right)^{\vee} \|_{\ell^p(\Z)} \lesssim \sup_{\omega \in X} |c_\omega| \cdot |X|^{|1/2 - 1/p|} \cdot \| f\|_{\ell^p(\Z)}\]
for each $1 < p < \infty$.
\end{lemma}
Here, the $\V^2(A)$ norm is defined as is the $\V^2(\RR)$ norm, with the restriction that all subsequences live inside $A$.

\subsection{A Technical Estimate}
This lemma exhibits one way in which pseudo-lacunarity is used. It is a generalization of the maximal multiplier estimate proved in \cite[Lemma 2.3]{KL}. Its proof follows the same entropy argument used there, so we choose to omit it. In practice, one will apply this lemma at many different scales as well as vary the value of $d$, which is why we need good control over the constants
\[ C_d := \sup_{0<\delta <1} \delta^d \mathcal{N}(\delta).\]

\begin{lemma}\label{Mink}
Suppose $1 \leq p < \infty$, $\Lambda \subset [0,1]$ has upper Minkowski dimension at most $d$, and set
\[ C_d := \sup_{0< \delta <1} \delta^d \mathcal{N}(\delta) ,\]
where $\mathcal{N}(\delta)$ is the $\delta$-entropy of $\Lambda$ as above.
Suppose that $\{ T_\lambda : \lambda \in [0,1] \}$ is a family of operators so that for each $f \in L^p$, $ T_\lambda f(x)$
is a.e.\ differentiable in $\lambda \in [0,1]$. Set
\[ a := \sup_{\lambda \in [0,1]} \| T_\lambda \|_{L^p \to L^p},\]
and
\[ A := \sup_{\lambda \in [0,1]} \| \partial_\lambda T_\lambda \|_{L^p \to L^p}.\]
Then
\[ \| \sup_\Lambda |T_\lambda f| \|_{L^p} \lesssim C_d^{1/p} (a + a^{1-d/p} A^{d/p} ) \|f\|_{L^p}.\]
\end{lemma}

\subsection{Variational Preliminaries}
For a sequence of functions, $\{ f_\lambda(x) \}$, we define the $r$-variation, $0<r< \infty$
\[ \V^r(f_\lambda)(x) := \sup \left( \sum_{i} | f_{\lambda_i} - f_{\lambda_{i+1}}|^r \right)^{1/r}(x),\]
where the supremum runs over all finite increasing subsequences $\{ \lambda_i\}$. (The $\infty$-variation, $\V^\infty(f_\lambda)(x) := \sup_{\lambda,\lambda'} |f_\lambda - f_{\lambda'}|(x)$ is comparable to the maximal function, $\sup_\lambda |f_\lambda|(x)$, and so is typically not introduced.) These variation operators are more difficult to control than the maximal function $\sup_\lambda |f_\lambda|$: for any $\lambda_0$, one may pointwise dominate
\[ \sup_\lambda |f_\lambda| \leq \V^\infty(f_\lambda) + |f_{\lambda_0}| \leq \V^r(f_\lambda) + |f_{\lambda_0}|,\]
where $r< \infty$ is arbitrary. This difficulty is reflected in the fact that although having bounded $r$-variation, $r < \infty$, is enough to imply pointwise convergence, there are functions which converge as $\lambda \to \infty$, but which have unbounded $r$ variation for any $r < \infty$. (e.g. $\{ f_\lambda = (-1)^\lambda \frac{1}{\log(1 + \lambda)} : \lambda \in \N \}$)

In developing an $L^2$-theory for his polynomial ergodic theorems, Bourgain proved and used crucially the following variational result \cite[Lemma 3.11]{B1}:
\begin{lemma} Let $A_\lambda f(x) := \frac{1}{\lambda} \int_0^\lambda f(x-t) \ dt$ denote the Lebesgue averaging operator. Then for $r > 2$,
\[ \| \V^r (A_\lambda f) \|_{L^2(\RR)} \lesssim \frac{r}{r-2} \| f\|_{L^2(\RR)}.\]
\end{lemma}

Since then, variational estimates have been the source of much work in ergodic theory and harmonic analysis. An important result for our paper is the following theorem, due to Oberlin, Seeger, Tao, Thiele, and Wright \cite[Theorem 1.2]{OSTTW}:
\begin{theorem}[Theorem 1.2 of \cite{OSTTW} -- Special Case]\label{var} Let $S_\lambda f(x) := \int_{-\infty}^\lambda \hat{f}(\xi) e(\xi x) \ d\xi$. Then for $r > 2$,
\[ \| \V^r(S_\lambda f) \|_{L^2(\RR)} \lesssim \frac{r}{r-2} \| f\|_{L^2(\RR)}.\]
\end{theorem}
\begin{remark}
The dependence on $r$ follows from close inspection of the proof of the stated theorem, \cite[page 24]{OSTTW}.
\end{remark}

In fact, our variational operators will never exactly involve the $\{ S_\lambda \}$ operators, but rather ``truncated" versions thereof,
\[ S^\phi_\lambda f(x) := \int_{-\infty}^\lambda \phi(\xi - \lambda) \hat{f}(\xi) e(\xi x) \ d\xi,\]
where $\phi \in C_c^\infty$ is compactly supported. The following convexity argument, kindly shared with us by Gennady Uraltsev, shows that the variation corresponding to $\{ S^\phi_\lambda\}$ may be controlled by the variation corresponding to the $\{ S_\lambda\}$:

\begin{lemma}\label{gen}
Suppose that $m(t)$ is a compactly supported absolutely continuous function.
then for any $0 < r < \infty$
\[ \V^r( \F^{-1}( m(\xi -\lambda) \hat{f}(\xi) ) ) \leq \| dm \|_{\text{TV}} \cdot \V^r (S_\lambda f), \]
pointwise; at the endpoint
\[ \sup_\lambda | \F^{-1}( m(\xi -\lambda) \hat{f}(\xi) ) | \leq \| dm \|_{\text{TV}} \cdot \sup_\lambda | S_\lambda f|. \]
 Here $\| - \|_{\text{TV}}$ denotes the total variation of the measure.
\end{lemma}
We will apply this lemma, of course, in the case where $m = \phi \cdot 1_{t < 0}$.
\begin{proof}
We begin with the variation.

With $x$ fixed, we select appropriate finite subsequences, $\{\lambda_i\} \subset \RR$ and $\{ a_i\} \subset \C$ with 
\[ \sum_i |a_i|^{r'} = 1 \]
so that
\[
\V^r( \F^{-1}( m(\xi -\lambda) \hat{f}(\xi) ) )(x) = \left| \sum_i \F^{-1} \left( (m(\xi - \lambda_{i+1}) - m(\xi-\lambda_i) ) \hat{f} (\xi) \right)(x) \cdot a_i \right|.\]

Writing
\[ m(\xi) = - \int 1_{[\xi,\infty)}(t) \ dm(t) = - \int 1_{(-\infty,t]}(\xi) \ dm(t)\]
and substituting appropriately allows us to express
\[ \aligned
&\left| \sum_i \F^{-1} \left( (m(\xi - \lambda_{i+1}) - m(\xi-\lambda_i) ) \hat{f} (\xi) \right)(x) \cdot a_i \right| \\
& \qquad =
\left| \int \sum_i 
\F^{-1} \left( 1_{(t+\lambda_i,t+\lambda_{i+1}]} \hat{f} (\xi) \right)(x) \cdot a_i \ dm(t) \right| \\
& \qquad \leq
\int \V^r (S_\lambda f)(x) \ d|m|(t) \endaligned\]
by the triangle inequality and the definition of the $\V^r$ norm.

The case of the supremum is easier: for any $\lambda, x$
\[ \aligned 
\left| \F^{-1} \left( m(\xi -\lambda) \hat{f}(\xi) \right) \right|(x) &= \left| \int \F^{-1} \left( 1_{ (-\infty,t+ \lambda] }(\xi) \hat{f}(\xi) \right)(x) \ dm(t) \right| \\
&\leq \int \sup_\lambda \left| \F^{-1} \left( 1_{(-\infty,\lambda]}(\xi) \hat{f}(\xi) \right) \right|(x) \ d|m|(t), \endaligned\]
which yields the result.

\end{proof}

\section{A Key Maximal Inequality}
In this section we prove the key maximal inequality of our paper. It is an extension of an inequality of Bourgain \cite[Lemma 4.11]{B1}, and follows closely his proof, using Lemma \ref{var} appropriately.

Let $P(t) := 1_{t < 0}$ and $1_{[-c,c]} \leq \phi \leq 1_{[-2c,2c]}$ be a Schwartz function, where $0 < c \ll 1$ is sufficiently small.
Set
\[ \widehat{S^s_\lambda f}(\xi) := P(\xi - \lambda) \phi(N^s(\xi - \lambda))\hat{f}(\xi),\]
where $N$ is a large enough dyadic integer, and $s \geq 0$ is an integer which will be fixed throughout.
By our variational convexity lemma, we know that
\[ \V^r (S^s_\lambda f ) \lesssim \V^r( S_\lambda f) \]
pointwise, uniformly in $s \geq 0$. 

Now, let $1_{[-a,a]} \leq \varphi \leq 1_{[-2a,2a]}$ be a Schwartz function, with $0 < c \ll a < 1$, and $a$ sufficiently larger than $c$. 

For a collection of frequencies $\{ \theta_1, \dots, \theta_K \} \subset \RR$, we say that $\{ \theta_i \}$ are \emph{$s$-separated} if
\[ \min_{1 \leq i \neq j \leq K} |\theta_i - \theta_j| > 2^{-2s-2}.\]

We have the following proposition, whose proof combines the argument of Bourgain and the variational result of \cite{OSTTW}:
\begin{proposition}\label{0}
For any collection of $s$-separated frequencies $\{ \theta_1, \dots, \theta_K \}$, we have the following estimate,
\[ Mf(x):= \sup_{|\lambda| \ll N^{-s}} \left| \sum_{i=1}^K e(\theta_i x) S^s_\lambda ( \varphi(N^s\xi) \hat{f}(\xi + \theta_i) )^{\vee}(x) \right| \]
is $L^2(\RR)$-bounded with operator norm $\lesssim \log^{2} K$.
\end{proposition}

Before beginning with the proof proper, we need some tools.

\subsection{Tools}
We will often work on the Hilbert space 
$\mathcal{H}:=L^{2}(\ell_K^2)$, equipped with the norm
\begin{equation}\label{e-normB}
\|\vec{f}\|_\mathcal{H}=\left\| \left( \sum_{i=1}^K |f_i|^2 \right)^{1/2} \right\|_{L^{2}(\RR)},\ \ 
\vec{f}=(f_1,\dots,f_K),
\end{equation}
and let the operators under consideration act on $\vec{f}\in\mathcal{H}$ componentwise, so that for example
\[ S_\lambda^s \vec{f} := (S_\lambda^s f_1,\dots, S_\lambda^s f_K).\]

With $\vec{f}$ given, define
\[ \aligned 
J_t(x) &:= J_t(S_\lambda^s \vec{f}(x) : \lambda) \\
&:= \sup \{ N : \exists \lambda_0 < \lambda_1 < \dots < \lambda_N \text{ with } \|S_{\lambda_i}^s \vec{f}(x) - S_{\lambda_{i-1}}^s \vec{f}(x)\|_{\ell^2_K} > t \} \endaligned \]
to be the {\it $t$-jump number} of $\{S_{\lambda_i}^s \vec{f}(x):\ \lambda\in\Lambda\}$ (note that the ``jumps'' are with respect to the parameter $\lambda$, with $x$ fixed).

For $\vec{f} \in \mathcal{H}$, we consider $\V^r_{\ell^2_K}(S_\lambda^s \vec{f})(x)$, with the $\ell^2_K$-norm replacing the Euclidean norm in the natural way. For each $0< t < \infty$, we have the pointwise inequality
\begin{equation}\label{vec-e20}
 t J_t(x)^{1/r} \leq \V^r_{\ell^2_K}(S^s_\lambda \vec{f})(x)
 \end{equation}
for each $0 < r < \infty$. Taking Lemma \ref{gen} into consideration, and using Minkowski's inequality, we are able to deduce that for each $r > 2$, 
\[ \V^r_{\ell^2_K} (S^s_\lambda \vec{f}) \leq \left( \sum_{i=1}^K \V^r (S^s_\lambda f_i)^2 \right)^{1/2} \lesssim \left( \sum_{i=1}^K \V^r (S_\lambda f_i)^2 \right)^{1/2}
\]
uniformly in $s$. Using also Theorem \ref{var}, it follows that
\begin{equation}\label{H}
 \| \V^r_{\ell^2_K} (S^s_\lambda \vec{f}) \|_{L^2(\RR)} \lesssim \frac{r}{r-2} \| \vec{f} \|_{\mathcal{H}}.
\end{equation}

We will need the following technical lemmas.

\begin{lemma}\label{vector-Carleson-lemma}
For any vector $\vec{f} = (f_1,\dots, f_K)$ with
\[ \left( \sum_{i=1}^K |f_i|^2 \right)^{1/2} \in L^p(\RR),\]
 define
\begin{equation}\label{def-F}
F(x) := \sup_\lambda \left( \sum_{i=1}^K |S^s_\lambda f_i(x)|^2 \right)^{1/2}. 
\end{equation}
Then for any $1 < p < \infty$,
\[ \|F\|_{L^{p}(\RR)}\lesssim \|\left( \sum_{i=1}^K |f_i|^2 \right)^{1/2} \|_{L^p(\RR)}.\]
\end{lemma}
\begin{proof}
By Lemma \ref{gen}, we may majorize
\[ F(x) \lesssim \sup_\lambda \left( \sum_{i=1}^K |S_\lambda f_i(x)|^2 \right)^{1/2} \]
uniformly in $s$. The result now follows from Hyt\"{o}nen's and Lacey's \cite[Theorem 1.1]{HL}.
\end{proof}

\begin{lemma}\label{tech,2}
For a fixed $\vec{f} \in \mathcal{H}$, define
the function $G:\RR\to\RR$,
\[ G(x)= \int \min\{ K^{1/2}, J_t(x)^{1/2} \} \ dt\]
Then $\|G\|_{L^2(\RR)}\lesssim \log^{2} K \cdot \| \vec{f} \|_{\mathcal{H}}$.
\end{lemma}
\begin{proof}
Define $F(x)$ as in (\ref{def-F}).
Note that for $t \geq 2 F(x)$, $J_t(x) = 0$ by the triangle inequality, so that 
\[ G(x)= \int_0^{2F(x)} \min\{ K^{1/2}, J_t(x)^{1/2} \} \ dt.\]
Choosing $r > 2$ so that
\[ \frac{1}{2} - \frac{1}{r} = \frac{1}{\log K},\]
we majorize 
\begin{align*}
 G(x)&\lesssim \int_0^{2F(x)/K^{1/2}} K^{1/2} \ dt + \int_{2F(x)/K^{1/2}}^{2F(x)} K^{1/2-1/r} \left( t \cdot J_t(x)^{1/r} \right) \ \frac{dt}{t}
 \\
  &\lesssim F(x) + \log K \cdot \V^r_{\ell^2_K}(S_\lambda^s \vec{f})(x),
 \end{align*}
 where at the last step we used (\ref{vec-e20}).
We now take the $L^2$ norms. The first term is acceptable by Lemma \ref{vector-Carleson-lemma}. To estimate the second term, we use that the $L^2(\RR)$-operator norm of $\V^r(S_\lambda \vec{f})$ grows like $\frac{r}{r-2}$ by (\ref{H}), so that the result follows.
\end{proof}

We will also use the following maximal lemma on $\ell^2_K$.
\begin{lemma}\label{max}
Suppose $A \subset \ell^2_K$ has finite cardinality $|A|$, and that $\{ \theta_1,\dots, \theta_K\}$ are $s$-separated. 
Then
\begin{align}
\label{eq:key:L2kMoment}
\bigg\| \sup_{a \in A} \bigg| \sum_{i=1}^K a_i e(\theta_i x) e(\theta_i u) \bigg| \bigg\|_{L^{2}_u[0,cN^s]} 
\lesssim N^{s/2} \cdot \min\{ K^{1/2}, |A|^{1/2} \} \sup_{a \in A} \|a \|_{\ell^2_K}.
\end{align}
%Here $0 < c \ll 1$ is a sufficiently small positive number.
\end{lemma}

\begin{proof}
For the $K^{1/2}$ bound, we just use Cauchy-Schwartz. For the cardinality bound, we use a change of variables and
bound the left-hand side of~\eqref{eq:key:L2kMoment} by
\begin{align*}
	N^{s/2} \left( \sum_{a \in A} \bigg\| \sum_{i=1}^K a_i e(\theta_i x) e(\lambda_i u ) \chi (u) \, \Bigg\|_{L^{2}_u(\mathbb{R})}^{2} \right)^{1/2},
\end{align*}
where $\chi$ is a smooth bump function equal to $1$ on $[-1,1]$
and the $\lambda_i$ are $1$-separated real numbers.
The desired bound then follows from Plancherel and the Fourier decay of $\chi$.
\end{proof}

\begin{lemma}\label{ent}
Suppose that $A \subset \ell^2_K$ is a finite subset, and represent elements $a \in A$ as 
$a = (a_1,\dots,a_K)$. Assume that the $\{\theta_i\}$ are $s$-separated.
Then for any $a \in A$, we have the upper bound
\[ \| \sup_{a \in A} \left| \sum_{i=1}^K e(\theta_i x) e(\theta_i u) a_i \right| \|_{L^{2}_u[0,cN^s]} \lesssim
N^{\frac{s}{2}} \cdot \sum_{l < r} 2^l \min\{ K^{1/2}, \mathcal{N}^*_A(2^l)^{1/2} \} + N^{\frac{s}{2}} |a|,\]
where $r$ is such that $2^{r-1} \leq \diam(A) < 2^r$, and $\mathcal{N}^*_A(t)$ denotes the minimum number of balls with radius $t$ centered at elements of $A$ required to cover $A$.\footnote{
This is of course comparable to the entropy number $\mathcal{N}_A(t)$.}
\end{lemma}

\begin{proof}
Suppose that $2^{r-1} \leq \diam(A) < 2^r$, and select an arbitrary $a^r \in A$. Then, let 
\[ \{ a^{r-1,j} : 1 \leq j \leq \mathcal{N}^*_A(2^{r-1})  \} \]
be such that
\[ A \subset \bigcup_j B(a^{r-1,j}, 2^{r-1}),\]
and $\{ a^{r-1,j} \}$ is minimal subject to this constraint. Inductively construct collections 
\[ \{ a^{l,j} : 1 \leq j \leq \mathcal{N}^*_A(2^l) \}\] subject to the same minimality condition, for all $l_0 \leq l < r$, where $l_0$ is so small that
\[ B(a,2^{l_0}) \cap A= \{ a\} \]
for each $a \in A$.

Define the implicit \emph{parent function} of a selected element as follows. For each $l_0 \leq l < r$, and each $1 \leq j \leq \mathcal{N}^*_A(2^l)$, define
\[ (a^{l,j})' \]
to be $a^{l+1,k}$ if 
\[ B(a^{l,j}, 2^l) \cap B(a^{l+1,k},2^{l+1}) \neq \emptyset,\]
and $k$ is the minimal index subject to the above constraint. Note that the parent of any $a^{r-1,j}$ is just $a^r$. Collect for each $l_0 \leq l < r$
\[ B_l := \{ a^{l,j} - (a^{l,j})' : 1 \leq j \leq \mathcal{N}^*_A(2^l)  \},\]
and note that
\[ \sup_{b \in B_l} |b| \lesssim 2^l\]
by the triangle inequality, while
\[ |B_l| \leq \mathcal{N}^*_A(2^l).\]

Now, for each $a \in A$, we have the telescoping representation
\[ a = (a-a') + (a' - a'') + \dots + a^r,\]
where each increment lives inside a particular set $B_l$, $l_0 \leq l < r$.

Consequently, we may majorize
\[ \sup_{a \in A} \left| \sum_{i=1}^K e(\theta_i x) e(\theta_i u) a_i \right| \leq
\sum_{l_0 =l}^{r-1} \sup_{b \in B_l} \left| \sum_{i=1}^K e(\theta_i x) e(\theta_i u) b_i \right| + 
\left| \sum_{i=1}^K e(\theta_i x) e(\theta_i u) a^r_i \right|.\]

Taking $L^{2}_u[0,c N^s]$ norms and applying our previous Lemma \ref{max} yields the desired estimate.
\end{proof}

We are now ready for the proof.

\subsection{The Proof}
Our goal is to prove Proposition \ref{0}, reproduced below for convenience:

\begin{proposition}
For any collection of $s$-separated frequencies $\{ \theta_1, \dots, \theta_K \}$, the operator $M$ defined by
\[ Mf(x):= \sup_{|\lambda| \ll N^{-s}} \left| \sum_{i=1}^K e(\theta_i x) S^s_\lambda ( \varphi(N^s\xi) \hat{f}(\xi + \theta_i) )^{\vee}(x) \right| \]
is $L^2(\RR)$-bounded with operator norm $\lesssim \log^{2} K$.
\end{proposition}

\begin{proof}[Proof of Proposition \ref{0}]
First, by continuity we may restrict $\lambda$ to a countable dense subset of $[-cN^{-s},cN^{-s}]$, and by monotone convergence we may restrict to a finite subcollection, $\lambda \in \Delta \subset [-cN^{-s},cN^{-s}]$, provided that our bounds are eventually independent of $\Delta$.

It is helpful to adopt a vector-valued perspective; define via the Fourier transform
\begin{equation}\label{vec-e2}
 \hat{f_i}(\xi) := \varphi(N^s \xi) \hat{f}(\xi + \theta_i) 
 \end{equation}
and note that $\{ f_1,\dots, f_K\}$ are pairwise orthogonal and have Fourier support inside $[-c_0N^{-s},c_0N^{-s}]$, for a suitably small $c_0$. We will then consider $M$ as an operator (for now, depending on $\Delta$) on 
the space $\mathcal{H}=L^{2}(\ell_K^2)$, 
defined by
\[
M\vec{f}(x)=\sup_{\lambda \in \Delta} \left| \sum_{i=1}^K e(\theta_i x) S^s_\lambda f_i(x) \right| ,\ \ 
\vec{f}=(f_1,\dots,f_K).
\]
We need to prove that for $f_1,\dots,f_k$ as in (\ref{vec-e2}),
\begin{equation}\label{vec-e1}
\| M\vec{f} \|_{L^{2}(\mathbb{R})}\lesssim \log^2 K \|\vec{f}\|_\mathcal{H}.
\end{equation}

Let $B$ denote the best constant in (\ref{vec-e1}); that is,
\[ B := \sup_{\|\vec{f}\|_\mathcal{B}=1, \supp \hat{f_i} \subset [-c_0N^{-s},c_0N^{-s}]}
 \| M\vec{f}  \|_{L^{2}(\RR)}.\]
We first claim that 
\begin{equation}\label{vec-e3}
B \lesssim K^{1/2} .
\end{equation}
This does not imply (\ref{vec-e1}) yet, but it will allow us to proceed later
with a bootstrapping argument. To prove (\ref{vec-e3}), we use Cauchy-Schwartz to write
\[ B \leq  \sup_{\|\vec{f}\|_\mathcal{B}=1, \ \supp \hat{f_i} \subset [-c_0N^{-s},c_0N^{-s}]}
K^{1/2} \| \sup_{\lambda} \left( \sum_{i=1}^K |S^s_\lambda f_i|^2 \right)^{1/2} \|_{L^{2}(\RR)},\]
and apply Lemma \ref{vector-Carleson-lemma}.

We now proceed with the proof of (\ref{vec-e1}). For any $u\in\RR$, we have
\[ \| M\vec{f}  \|_{L^{2}(\RR)}\leq  \| M (T_u \vec{f} ) \|_{L^{2}(\RR)}+ \| M (\vec{f} -T_u \vec{f}) \|_{L^{2}(\RR)},\]
 where
 \[
 T_u\vec{f} (x)=(f_1(x-u),\dots, f_K(x-u)).
 \]
 The strategy of the proof is as follows. We will prove that if  $c>0$ is sufficiently small, then

\medskip
(a) for all $u\in[0,c2^s]$, we have
$\|\vec{f} -T_u \vec{f}\|_\mathcal{H}\leq \frac{1}{2}\|\vec{f}\|_\mathcal{H} $,

\medskip
(b) there is a $u\in[0,c2^s]$ such that $\| M (T_u \vec{f} ) \|_{L^{2}(\RR)}\lesssim \log^2 K \|\vec{f}\|_\mathcal{H}$.

\medskip

Let $\vec{f}\in\mathcal{H}$ with $\|\vec{f}\|_\mathcal{H}=1$ and $\supp \hat{f_i} \subset [-c_0N^{-s},c_0N^{-s}]$, and choose $u$ (possibly depending on $\vec{f}$) 
such that both (a) and (b) hold. Then for some $C = O(1)$
\[ \| M\vec{f}  \|_{L^{2}(\RR)}\leq  C \log^2 K + \frac{1}{2} B . \]
Taking the supremum over all admissible $\vec{f}$, we get that 
\[ B \leq C \log^2 K + \frac{1}{2} B,\]
which proves (\ref{vec-e1}).

We first prove (a). This is a direct consequence of Plancherel's inequality, taking into account the Fourier support restriction on $f_i$. Indeed, let $\tilde\varphi$ be a smooth function such that $\tilde\varphi\equiv 1$ on $[-c_0N^{-s},c_0N^{-s}]$
and $\supp \tilde\varphi \subset [-2c_0N^{-s},2c_0N^{-s}]$.
Then
for each $f_i$,
\[ \aligned
\| f_i(x) - f_i(x-u) \|_{L^2(\RR)} &= \| (1-e(-\xi u)) \cdot \tilde\varphi(N^s \xi) \hat{f}(\xi + \theta_i) \|_{L^2(\RR)} \\
&\lesssim
\sup_{|\xi| \leq 2c_0N^{-s}} |\xi u | \cdot \| \tilde\varphi(N^s \xi) \hat{f}(\xi+ \theta_i) \|_{L^2(\RR)} \\
&\ll \frac{1}{2} \|f_i\|_{L^2(\RR)}, \endaligned\]
for $|u| \leq cN^s$ with $c$ sufficiently small.
 
To prove (b), it suffices to show that it holds on average, i.e.,
\[ \aligned 
& N^{- \frac{s}{2}} \| \| \sup_{\lambda \in \Delta} \left| \sum_{i=1}^K e(\theta_i x) S^s_\lambda f_i(x-u) \right| \|_{L^{2}_u[0,cN^s]} \|_{L^{2}_x(\RR)} \\
& \qquad = 
N^{- \frac{s}{2}} \| \| \sup_{\lambda \in \Delta } \left| \sum_{i=1}^K e(\theta_i u) e(\theta_i x) S^s_\lambda f_i(x) \right| \|_{L^{2}_u[0,cN^s]} \|_{L^{2}_x(\RR)} \endaligned \]
is bounded by a constant multiple of 
\[ \log^2 K \| \vec{f} \|_{\mathcal{H}}.\]

With $x \in \RR$ fixed, we consider the set
\[ A = A_x = \{ S^s_\lambda f_1(x), \dots, S^s_\lambda f_K(x) : \lambda \in \Delta \} \subset \ell^2_K.\]
The quantity $\mathcal{N}^*_A(t)$ (defined in Lemma \ref{ent}) is bounded by $J_t(x)+1$; in
particular, for any $t < \diam(A)$ we have $\mathcal{N}^*_A(t) \lesssim J_t(x)$.
Applying Lemma \ref{ent}, we majorize the inner integral by a constant multiple of
\[ \int_0^{2F(x)} \min\{ K^{1/2}, J_t(x)^{1/2} \} \ dt  + \left( \sum_{i=1}^K |S^s_{\lambda_0} f_i(x)|^2 \right)^{1/2} \]
where $\lambda_0 \in \Delta$ is arbitrary, and 
\[ F(x) := \sup_\lambda \left( \sum_{i=1}^K |S^s_\lambda f_i|^2 \right)^{1/2} \]
is as above.

Taking $L^{2}(\RR)$ norms and applying Lemma \ref{tech,2} now completes the proof.
\end{proof}

Finally, we have the following straightforward result, which we will use for interpolation purposes in $\S 5$ below:

\begin{proposition}\label{2}
For any $1 < p < \infty$, and any collection of frequencies $\{ \theta_1, \dots, \theta_K \}$, we have the following estimate,
\[ \sup_{|\lambda| \ll N^{-s}} \left| \sum_{i=1}^K e(\theta_i x) S^s_\lambda ( \varphi(N^s\xi) \hat{f}(\xi + \theta_i) )^{\vee}(x) \right| \]
is $L^p(\RR)$-bounded with operator norm $\lesssim K^{\theta(p)}$, where
\[ \theta(p) := \begin{cases} \frac{1}{p} &\mbox{if } 1 < p \leq 2  \\ 
\frac{1}{2} & \mbox{if } 2 \leq p < \infty. \end{cases}\]
\end{proposition}
\begin{proof}
We use Cauchy-Schwartz in $1 \leq i \leq K$ and then Lemma \ref{gen} to estimate
\[
\sup_{|\lambda| \ll N^{-s}} \left| \sum_{i=1}^K e(\theta_i x) S^s_\lambda ( \varphi(N^s\xi) \hat{f}(\xi + \theta_i) )^{\vee} \right| \lesssim K^{1/2} F(x),\]
where 
\[ F(x) := \sup_{\lambda} \left( \sum_{i=1}^K | S_\lambda \left( \varphi(N^s\xi) \hat{f}(\xi+\theta_i) \right) |^2 \right)^{1/2}(x).\]
Taking $L^p(\RR)$ norms and applying Lemma \ref{vector-Carleson-lemma} shows that the maximal function is bounded in norm by
\[ K^{1/2} \| \left( \sum_{i=1}^K | \F^{-1} \left( \varphi(N^s(\xi - \theta_i)) \hat{f}(\xi) \right) |^2 \right)^{1/2} \|_{L^p(\RR)},\]
by an easy change of variables. By Rubio de Francia's inequality \cite{R} (see also \cite{L} for a nice exposition), we have
\[ \| \left( \sum_{i=1}^K | \F^{-1} \left( \varphi(N^s(\xi - \theta_i)) \hat{f}(\xi) \right) |^2 \right)^{1/2} \|_{L^p(\RR)} \lesssim \|f\|_{L^p(\RR)} \]
for $2 \leq p < \infty$.
By the argument of \cite[Proposition 4.3]{K}, we also have that
\[ \| \left( \sum_{i=1}^K | \F^{-1} \left( \varphi(N^s(\xi - \theta_i)) \hat{f}(\xi) \right) |^2 \right)^{1/2}  \|_{L^{1,\infty}(\RR)} \lesssim K^{1/2} \|f\|_{L^1(\RR)}, \]
which implies that
\[ \| \left( \sum_{i=1}^K | \F^{-1} \left( \varphi(N^s(\xi - \theta_i)) \hat{f}(\xi) \right) |^2 \right)^{1/2}  \|_{L^{p}(\RR)} \lesssim K^{1/p - 1/2} \|f\|_{L^p(\RR)}, \]
for $1 < p \leq 2$ by interpolation. Multiplying by $K^{1/2}$ completes the result.
\end{proof}

\begin{remark} A natural question to ask is the correct operator-norm growth of the ``multi-frequency Carleson operator''
\[ \sup_\lambda \left| \sum_{i=1}^K e(\theta_i x) S^s_\lambda ( \varphi(N^s\xi) \hat{f}(\xi + \theta_i) )^{\vee} (x) \right|, \]
where $\{\theta_1,\dots,\theta_K\}$ are $s$-separated. Stronger estimates on this operator will lead to an improved range of $\ell^p(\Z)$ estimates in Theorem \ref{Main}, see $\S 5$ below.
We anticipate returning to this problem in the future.
\end{remark}

\section{Reduction to Multi-Frequency Carleson}

We now reduce the proof of Theorem \ref{Main} to estimating a multi-frequency Carleson operator of the form considered in Section 3. We first note that Theorem \ref{Main} can be rewritten as a maximal multiplier problem. Specifically, we have
\begin{align*}
\Ca_\Lambda f(n) &= \sup_{\lambda \in \Lambda} \left| \sum_{p \in \pm \PP} f(n-p) \log |p| \frac{e(\lambda p)}{p} \right|\\
 &= \sup_{\lambda \in \Lambda} \left| f*\mathcal{K}_\lambda (n) \right|,
\end{align*}
with
 \[
 \mathcal{K}_\lambda=  \sum_{p \in \pm \PP} \log |p| \frac{e(\lambda p)}{p}
 \delta_p.
 \]
 Using Fourier transforms, we rewrite this as 
 \[
 \Ca_\Lambda f(n) = \sup_{\lambda \in \Lambda}  \left| \mathcal{F}^{-1}
 (m(\lambda,\cdot) \widehat{f})(n) \right|,
\]
where
 \[ m(\lambda,\beta) := \widehat{\mathcal{K}_\lambda}(\beta) =\sum_{p \in \pm \PP} \frac{e(\lambda p -\beta p) \log|p|}{p} .\]

\subsection{Number-Theoretic Approximations}

Following Mirek-Trojan \cite[\S 3]{MT}, we first introduce an intermediate number-theoretic approximation. 
Throughout this section and the following one, $1 \ll \alpha \lesssim_p 1$ will be a sufficiently large but fixed constant.

We begin by dyadically decomposing 
\[ m(\lambda,\beta) := \sum_{j\geq 2} m_j(\lambda,\beta) := \sum_{j \geq 2} \left( \sum_{p \in \pm \PP} e(\lambda p -\beta p) \log|p| \psi_j(p) \right),\]
where $\sum_{j \geq 2} \psi_j(t) = \frac{1}{t}$ for $|t| \geq 2$.
Each $m_j(\lambda,\beta)$ is the Fourier transform (with $n$ and $\beta$ as the dual variables) of the kernel
\[ K_j(\lambda,n) := \sum_{p \in \pm \PP} e(\lambda p) \log |p| \psi_j(p) \delta_p(n).\]

We consider the approximating multipliers 
\[ \nu_j(\beta) := \sum_{s \geq 0 : 2^s < \frac{2^{j-1}}{j^{2\alpha}}} \nu_j^s(\beta) :=
\sum_{s \geq 0 : 2^s < \frac{2^{j-1}}{j^{2\alpha}}} \sum_{\R_s} \frac{\mu(q)}{\varphi(q)} \widehat{\psi_j}(\beta - a/q) \chi_s( \beta - a/q). \]
These are identical to the multipliers from \cite[Proposition 3.2]{MT}, except for
the cut-off $2^s < \frac{2^{j-1}}{j^{2\alpha}}$ in $s$.
Here, $\R_s := \{ a/q \in \TT \cap \QQ : (a,q) = 1, 2^s \leq q < 2^{s+1} \}$, $\mu$ is the M\"{o}bius function, $\varphi$ is the totient function,
$\chi$ is a Schwartz function satisfying
\[ 1_{[-c/N,c/N]} \leq \chi \leq 1_{[-2c/N,2c/N]},\]
and 
\[ \chi_s(t) := \chi(N^st)\]
where $N$ is our large enough dyadic integer. We will often use the well known estimate on the growth of the totient function: for any $\epsilon > 0$, there exists an absolute constant $C_\epsilon > 0$ so that
\begin{equation}\label{phi-growth}
 \varphi(q) \geq C_{\epsilon} q^{1-\epsilon}.
 \end{equation}
We also note here that the multipliers $\nu_j$ depend on the value of the constant $\alpha$ introduced above. Since $\alpha$ will be fixed throughout the rest of this paper, we will not display that dependence.

We then have the following.

\begin{proposition}\label{app}
For every $\alpha > 16$,
\[ |E_j(\lambda,\beta)| := |m_j(\lambda,\beta) - \nu_j(\beta- \lambda)| \lesssim_\alpha j^{-\alpha/4}.\]
\end{proposition}

For the similar multipliers without the cut-off in $s$, this is Proposition 3.2 of \cite{MT}. The proof of that proposition can be modified easily to produce Proposition \ref{app}. We omit the details.

\subsection{Cleaning Up the Multipliers}
Our first goal now is to approximate $m(\lambda,\beta) = \sum_{j\geq 2} m_j(\lambda,\beta)$, acting as a multiplier on $\ell_p(\mathbb{Z})$, by $\sum_{j\geq 2} \nu_j(\lambda,\beta)$. To this end, we introduce the  ``error function"
\[ \mathcal{E}f(n) := \sup_{\lambda \in \Lambda} \left| \left( \sum_{j\geq 2} E_j(\lambda,\beta) \hat{f}(\beta) \right)^{\vee} \right|(n) \]
in $\ell^p(\Z)$, with $E_j(\lambda,\beta)$ defined in Proposition \ref{app}.

\begin{lemma}\label{error-est}
For any $1<p< \infty$, $\mathcal{E}f$ is bounded on $\ell^p(\Z)$, provided $\alpha$ is taken sufficiently large.
\end{lemma}
\begin{proof}
We first claim that  for any $0 \leq \lambda \leq 1$,
\begin{equation}\label{nt-e1}
 \| \left( \nu_j(\beta - \lambda) \hat{f}(\beta) \right)^{\vee} \|_{\ell^p(\Z)} \lesssim \|f \|_{\ell^p(\Z)},
 \end{equation}
and
\begin{equation}\label{nt-e2}
 \| \left( \partial_\lambda \nu_j(\beta - \lambda) \hat{f}(\beta) \right)^{\vee} \|_{\ell^p(\Z)} \lesssim N^j \|f \|_{\ell^p(\Z)},
 \end{equation}
for each $1< p < \infty$.
We first consider (\ref{nt-e1}) with $\nu_j$ replaced by $\nu_j^s$ for a fixed $s$. 
To this end, we apply Lemma \ref{KK} with $|X|=|\mathcal{R}_s|\lesssim  2^{2s}$. Using that 
\[ \aligned 
\|\widehat{\psi_j}  \chi_s \|_{\mathcal{V}^2(\RR)}  &\leq \| \partial( \widehat{\psi_j}  \chi_s ) \|_{L^1(\RR)} \\
&\leq \| \partial (\widehat{\psi_j} ) \chi_s \|_{L^1(\RR)} + \| \widehat{\psi_j} \partial (\chi_s ) \|_{L^1(\RR)} \\
& \lesssim 1 \endaligned \]
uniformly in $s$ and $j$, and that $|\varphi(q)|\geq C_\epsilon 2^{(1-\epsilon)s}$ for $q\geq 2^s$ by (\ref{phi-growth}), we get that
\begin{align} \| \left(  \nu_j^s(\beta - \lambda) \hat{f}(\beta) \right)^{\vee} \|_{\ell^p(\Z)} 
&\lesssim 2^{s(2|\frac{1}{2}-\frac{1}{p}|-(1-\epsilon))} \|f \|_{\ell^p(\Z)}.
\end{align}
Choose $\epsilon>0$ small enough so that 
$2|\frac{1}{2}-\frac{1}{p}|-(1-\epsilon)<0$ (note that $\epsilon$ here depends only on $p$ and not on $s$). This proves (\ref{nt-e1}) for $\nu_j^s$; since, moreover, the bounds we just obtained are summable in $s$, we get (\ref{nt-e1}) for $\nu_j$. The proof of (\ref{nt-e2}) is identical, except that we use the estimate
\[ \aligned
\| \partial ( \widehat{ \psi_j} \chi_s) \|_{\V^2(\RR)} &\leq \| \partial^2 ( \widehat{ \psi_j} \chi_s) \|_{L^1(\RR)} \\
&\leq \| \partial^2 ( \widehat{ \psi_j} ) \chi_s  \|_{L^1(\RR)} + \| \partial ( \widehat{ \psi_j}) \partial( \chi_s) \|_{L^1(\RR)} + \| \widehat{ \psi_j} \partial^2 (  \chi_s) \|_{L^1(\RR)} \\
&\lesssim 2^{j} + 2^j + N^s \\
&\lesssim N^j \endaligned\]
uniformly in $s \geq 0$ such that $2^s < \frac{2^{j-1}}{j^{2\alpha}}$ and $j$.

Next, we bound
\[ \| \left( \partial_\lambda^i m_j(\lambda,\beta) \hat{f}(\beta) \right)^{\vee} \|_{\ell^p(\Z)} \leq \| \partial^i_\lambda K_j(\lambda,n) \|_{\ell^1(\Z)} \|f\|_{\ell^p(\Z)} \lesssim 2^{ji} \|f\|_{\ell^p(\Z)}, \; i =0,1,\]
by Young's convolution inequality and the prime number theorem; both bounds are uniform in $0 \leq \lambda \leq 1$.

Putting everything together, we observe that for each $1 < p < \infty$,
\[ \aligned 
&\sup_\lambda \| \left( (E_j(\lambda,\beta) \hat{f}(\beta) \right)^{\vee} \|_{\ell^p(\Z)} \\
& \qquad \leq
\sup_\lambda \| \left( m_j(\lambda,\beta) \hat{f}(\beta) \right)^{\vee} \|_{\ell^p(\Z)} + \sup_\lambda \| \left( \nu_j(\beta - \lambda) \hat{f}(\beta) \right)^{\vee} \|_{\ell^p(\Z)} \\
& \qquad \lesssim \|f\|_{\ell^p(\Z)}, \endaligned \]
and
\[ \aligned
& \sup_\lambda \| \left( \partial_\lambda E_j(\lambda,\beta)  \hat{f}(\beta) \right)^{\vee} \|_{\ell^p(\Z)} \\
& \qquad \leq
\sup_\lambda \| \left( \partial_\lambda m_j(\lambda,\beta) \hat{f}(\beta) \right)^{\vee} \|_{\ell^p(\Z)} + \sup_\lambda \| \left( \partial_\lambda \nu_j(\beta - \lambda) \hat{f}(\beta) \right)^{\vee} \|_{\ell^p(\Z)} \\
& \qquad \lesssim N^j \|f\|_{\ell^p(\Z)}. \endaligned \]

Next, we use Proposition \ref{app} to improve the first inequality above. For $p = 2$, we have by Proposition \ref{app} 
\[ 
\sup_\lambda \| \left( (E_j(\lambda,\beta) \hat{f}(\beta) \right)^{\vee} \|_{\ell^2(\Z)} \lesssim
 j^{-\alpha/4} \|f\|_{\ell^2(\Z)}.\]
Therefore by interpolation,
\[ \sup_\lambda \| \left( (E_j(\lambda,\beta) \hat{f}(\beta) \right)^{\vee} \|_{\ell^p(\Z)} \lesssim
j^{-\delta_p} \| f \|_{\ell^p(\Z)},\]
where $\delta_p = \delta_p(\alpha) \gg 1$. Note that for a fixed $p$, $\delta_p\to\infty$ as $\alpha\to\infty$.

We are now ready to deduce our error term estimate. We have
\[ \aligned
\| \mathcal{E}f \|_{\ell^p(\Z)} &\leq \sum_{j\geq 2} \| \sup_{\lambda \in \Lambda} \left| \left( E_j(\lambda, \beta) \hat{f}(\beta) \right)^{\vee}\right| \|_{\ell^p(\Z)} \\
&\lesssim \sum_{j\geq2} C_{1/j}\left( j^{-\delta_p} + j^{-\delta_p (1-\frac{1/j}{p})} N^{j \cdot \frac{1/j}{p}} \right) \| f\|_{\ell^p(\Z)}, \endaligned \]
by Lemma \ref{Mink} with $d=1/j$, where
\[ C_{1/j} := \sup_{0<\delta <1} \delta^{1/j} \mathcal{N}(\delta),\]
and $\mathcal{N}(\delta) = \mathcal{N}_\Lambda(\delta)$ is the smallest number of intervals of length $\delta$ needed to cover $\Lambda$.
By our pseudo-lacunarity assumption on $\Lambda$, we know there exists some absolute $M\lesssim 1$ so that
\[ C_{1/j} \lesssim j^M\]
for each $j$. Inserting this estimate into the foregoing leads to a bound on each term in the sum of
\[ j^{M - \delta_p/2};\]
provided that $\alpha$ has been chosen sufficiently large, this sums nicely, proving the lemma.
\end{proof}

\subsection{Reduction to Carleson}
By Lemma \ref{error-est}, we have reduced the proof of Theorem \ref{Main} to estimating the maximal function
\[ \aligned
&\sup_{\lambda \in \Lambda} \left| \left( \sum_{j\geq 2} \nu_j(\beta - \lambda) \hat{f}(\beta) \right)^{\vee} \right|  \\
&\qquad \leq  \sum_{s\geq0} \sup_{\lambda \in \Lambda} \left| \left( \sum_{j : \frac{2^j}{j^{2\alpha}} > 2^{s+1}} \nu_j^s(\beta - \lambda) \hat{f}(\beta) \right)^{\vee} \right| \\
& \qquad=: \sum_{s\geq0} \sup_{\lambda \in \Lambda} \left| \left( \sum_{\R_s} \frac{\mu(q)}{\varphi(q)} \Psi^s(\beta - \lambda - a/q) \chi_s (\beta - \lambda - a/q) \times \hat{f}(\beta) \right)^{\vee} \right|,  \endaligned\]
 in $\ell^p(\Z)$, where we use the notation
\[ \Psi^s(\beta) := \sum_{j :\frac{2^j}{j^{2\alpha}} > 2^{s+1}} \widehat{\psi_j}(\beta).\]
Since for each scale $s \geq 0$ our multipliers are compactly supported inside $\TT$, we may apply Lemma \ref{trans} and view them instead as real-variable multipliers. Indeed, restricting our set of modulation parameters to $\Lambda_T \subset \Lambda$ an arbitrary (finite) $T$-element subset of $\Lambda$, if we knew that
\[ \sup_{\lambda \in \Lambda_T} \left| \left( \sum_{\R_s} \frac{\mu(q)}{\varphi(q)} \Psi^s(\xi - \lambda - a/q) \chi_s (\xi - \lambda - a/q) \times \hat{f}(\xi) \right)^{\vee} \right|\]
were $L^p(\RR)$ bounded with operator norm $A_s$, by invoking Lemma \ref{trans} with $B_1 = \C$ and $B_2 = (\C^T, \| - \|_{\ell^\infty_T})$, we would be able to deduce that the analogously restricted operator on the integers would be similarly $\ell^p(\Z)$ bounded with operator norm $\lesssim A_s$; invoking monotone convergence would yield the result for the unrestricted maximal operator. Here, we used that our set of modulation parameters may be assumed to be countable, since $\Z$ itself is countable.

So motivated, we will shift our perspective to the real-variable setting, with the aim of proving that
\[ A_s \lesssim 2^{-\kappa_p s}\]
for some $\kappa_p > 0$.

Our final reduction in this section is the following.

\begin{proposition}\label{carleson-red-prop}
In order to prove Theorem \ref{Main}, it suffices to prove that, for the same range of exponents $p$, we have
\[
\|\Ca_{\Lambda}^s f \|_{L^p(\RR) } \lesssim 2^{-\kappa_p s}\|f\|_{L^p(\RR)}
\]
for some $\kappa_p > 0$, where
\[ \Ca_{\Lambda}^s f := \sup_{\lambda \in \Lambda} \left| \left( \sum_{\R_s} \frac{\mu(q)}{\varphi(q)} P(\xi - \lambda - a/q) \chi_s (\xi - \lambda - a/q) \times \hat{f}(\xi) \right)^{\vee} \right|.\]
\end{proposition}

\begin{proof}
The idea is to ``complete" $\Psi^s(\xi)$ to 
\[ P(\xi) := 1_{\xi < 0} = \frac{1}{2} - \frac{\sgn(\xi) }{2},\]
and prove that the difference is bounded on $L^p(\RR)$.

To this end, we define 
\[ \Psi_s(\xi) := \sum_{j \in \Z :\frac{2^j}{j^{2\alpha}} \leq 2^{s+1}} \widehat{\psi_j}(\xi),\]
so that
\[ \Psi_s(\xi) + \Psi^s(\xi) = - \pi i \cdot \sgn(\xi).\]
More explicitly, if $1_{[-1/2,1/2]} \leq \eta \leq 1_{[-1,1]}$ is the Schwartz function from $\S 2.1$, and 
\[ j(s) := \max\{ j \in \Z : \frac{2^j}{j^{2\alpha}} \leq 2^{s+1} \},\]
then we may write
\[ \Psi_s(\xi) = -\pi i \sgn * \left( 2^{j(s)}\hat{\eta}(2^{j(s)} -)\right)(\xi) \]
and 
\[\Psi^s(\xi) = -\pi i \sgn *\left( \delta - \left( 2^{j(s)}\hat{\eta}(2^{j(s)} -) \right) \right)(\xi), \]
where $\delta$ is the point mass at the origin.
We consider the multipliers
\[ E^s_1(\lambda,\xi) := \sum_{\R_s} \frac{\mu(q)}{\varphi(q)} \Psi_s(\xi - \lambda - a/q) \chi_s (\xi - \lambda - a/q) \]
and
\[ E^s_2(\lambda,\xi) := \sum_{\R_s} \frac{\mu(q)}{\varphi(q)} \chi_s (\xi - \lambda - a/q). \]

By Lemma \ref{K1}, for each $i= 1,2$, and $1 < p < \infty$, taking the $\epsilon$ in the totient function estimate (\ref{phi-growth}), we may find $\delta_p > 0$ so that
\begin{equation}\label{vec-e30} 
\begin{split}
\sup_{0 \leq \lambda \leq 1} \| \left( E^s_i(\lambda,\xi) \hat{f}(\xi) \right)^{\vee} \|_{L^p(\RR)} &\lesssim 2^{(\epsilon-1)s +2s|1/p-1/2|} \|f\|_{L^p(\RR)}\\
& = 2^{-\delta_p s} \|f\|_{L^p(\RR)}.\\
\end{split}
\end{equation}
Indeed, one takes $|X| = |\R_s| \lesssim 2^{2s}$, and uses that
\[ \aligned 
\| \Psi_s \chi_s \|_{\V^2(\RR)} + \| \chi_s\|_{\V^2(\RR)} &\leq \| \Psi_s \chi_s \|_{\V^1(\RR)} + \| \chi_s \|_{\V^1(\RR)} \\
&\leq \| \Psi_s \|_{\V^1(\RR)} \| \chi_s \|_{L^\infty(\RR)} + \left( \| \Psi_s \|_{L^\infty(\RR)} + 1 \right) \| \chi_s \|_{\V^1(\RR)} \\
&\lesssim 1, \endaligned \]
since 
\[ \| fg \|_{\V^1(\RR)} \leq \| f \|_{\V^1(\RR)} \|g\|_{L^\infty(\RR)} + \|f \|_{L^\infty(\RR)} \| g\|_{\V^1(\RR)},\]
and $\| \Psi_s \|_{\V^1(\RR)}$ may be estimated as follows:
\[ \aligned 
\left| \sum_{i} \Psi_s(\xi_i) - \Psi_s(\xi_{i+1}) \right| &\lesssim \left| \int \left( \sum_{i} \sgn(\xi_i - \zeta) - \sgn(\xi_{i+1} - \zeta) \right) 2^{j(s)} \hat{\eta}(2^{j(s)} \zeta) \ d\zeta \right| \\
&\leq \| \sgn \|_{\V^1(\RR)} \| \hat{\eta} \|_{L^1(\RR)} \\
&\lesssim 1. \endaligned\]
Arguing similarly, one further computes
\[ \aligned
 \sup_{0 \leq \lambda \leq 1} \| \left( \partial_\lambda E^s_i(\lambda,\xi) \hat{f}(\xi) \right)^{\vee} \|_{L^p(\RR)} &\lesssim N^s \cdot 2^{(\epsilon-1)s +2s|1/p-1/2|} \|f\|_{L^p(\RR)} \\
&= N^s \cdot 2^{-\delta_p s} \|f\|_{L^p(\RR)}; \endaligned \]
the key computation here is that for an appropriate finite subsequence $\{ \xi_i \} \subset \RR$
\[ \aligned 
\| \partial \Psi_s \|_{\V^1(\RR)} &\lesssim \left| \int \left( \sum_{i} \sgn(\xi_i - \zeta) - \sgn(\xi_{i+1} - \zeta) \right) 2^{2j(s)} (\partial \hat{\eta})(2^{j(s)} \zeta) \ d\zeta \right| \\
&\leq \| \sgn \|_{\V^1(\RR)} \cdot 2^{j(s)} \| \partial \hat{\eta} \|_{L^1(\RR)} \\
&\lesssim 2^{j(s)}. \endaligned \]

By Lemma \ref{Mink}, taking $d > 0$ sufficiently small, we may estimate the $L^p(\RR)$ operator norm of 
\[ \| \sup_{\lambda \in \Lambda} | \left( E^s_i(\lambda,\xi) \hat{f}(\xi) \right)^{\vee}  \|_{L^p(\RR)} \lesssim
C_d^{1/p}( 2^{-\delta_p s} + 2^{-\delta_p s } N^{s \frac{d}{p}} ) \lesssim 2^{-\delta_p' s},\]
for some $\delta_p > \delta_p' >0$. Since this error is geometrically decaying, and thus summable, in $s \geq 0$, and $P$ is a linear combination of $\Psi_s, \Psi^s$, and the identity function, the proposition follows.
\end{proof}

\section{Completing the Proof}
In the last section we reduced matters to proving that the maximal function $\Ca_{\Lambda}^s f$
is bounded on $L^p(\RR)$ with norm $2^{-\kappa_p s}$ for some $\kappa_p > 0$.

The first move will be to split $\Ca_{\Lambda}^s$ into $\lesssim 2^{\epsilon s}$ separate sub-maximal functions, each of which is closer to the maximal functions considered in $\S 3$. To do this, we introduce a covering of $\Lambda$,
\[ \Lambda \subset \bigcup_{j=1}^{\mathcal{N}(2cN^{-s})} I_j = \bigcup_{j=1}^{\mathcal{N}(2cN^{-s})} [a_j - cN^{-s},a_j + cN^{-s}] \]
using $\mathcal{N}(2cN^{-s}) \lesssim 2^{\epsilon s}$ intervals. By an easy change of variables $\lambda\to\lambda+a_j$, we majorize
\[ 
\Ca_\Lambda^s f  \leq \sum_{j \lesssim 2^{\epsilon s}} 
\sup_{|\lambda| \ll N^{-s}} \left| \left( \sum_{\R_s} \frac{\mu(q)}{\varphi(q)} P(\xi - \lambda - a_j - a/q) \chi_s (\xi - \lambda - a_j - a/q) \times \hat{f}(\xi) \right)^{\vee} \right|.
\]
Note that the shifting $\R_s \ni a/q \to a/q + a_j$ does not affect the separation properties of the $a/q$. The shifted numbers might no longer be in $\mathcal{R}_s$, but this will not be necessary at this stage.

If we let (say)
\[ 1_{[-4c/N,4c/N]} \leq \tilde{\chi}(x) \leq 1_{[-8c/N,8c/N]} \]
be another Schwartz function, and 
\[ \tilde{\chi}_s(x) := \tilde{\chi}(N^s x),\]
for any $s$-separated collection of frequencies $\{\theta_1,\dots,\theta_K\}$, and any $|\lambda| \ll N^{-s}$, we can write the multiplier
\[ \sum_{k=1}^K c_k P(\xi - \lambda -\theta_k) \chi_s(\xi-\lambda -\theta_k) \]
as a product
\[ \sum_{k=1}^K P(\xi - \lambda -\theta_k) \chi_s(\xi-\lambda -\theta_k) \times  \sum_{j=1}^K c_j \tilde{\chi}_s(\xi - \theta_j), \]
since the sums diagonalize by our $s$-separation condition on the frequencies. In particular, if we write
\[ M_{\{c_j\},s} \hat{f}(\xi) := \sum_{j=1}^K c_j \tilde{\chi}_s(\xi-\theta_j) \hat{f}(\xi),\]
we may express
\[
\sup_{|\lambda | \ll N^{-s}} \left| \F^{-1}\left( \sum_{k=1}^K c_k P(\xi - \lambda -\theta_k) \chi_s(\xi-\lambda -\theta_k) \hat{f}(\xi) \right) \right| \]
as
\[ \sup_{|\lambda| \ll N^{-s}} \left| \F^{-1} \left( \sum_{k=1}^K P(\xi - \lambda -\theta_k) \chi_s(\xi-\lambda -\theta_k) \times M_{\{c_j\},s} \hat{f}(\xi) \right) \right|. \]
We will refer to this observation as the ``lifting" trick.

\begin{lemma}\label{L2}
The maximal function $\Ca_{\Lambda}^sf$ obeys
\[
\|\Ca_{\Lambda}^s f\|_{L^2(\RR)} \lesssim_\epsilon  2^{(\epsilon -1)s}\|f\|_{L^2(\RR)}
\]
for any $\epsilon > 0$.
\end{lemma}
\begin{proof}
By the triangle inequality, it is enough to treat the maximal function with $|\lambda | \ll N^{-s}$ and all the frequencies shifted by some $a_j$.

We apply the lifting trick with $\{\theta_1,\dots,\theta_K\}= a_j + \R_s$, and $c_k := \frac{\mu(q)}{\varphi(q)}$ if $\theta_k = a_j + a/q$. The operator norm of the multiplier $M_{\{c_j\},s}$ is $\lesssim 2^{(\epsilon-1)s}$ by Plancherel and (\ref{phi-growth}); the operator norm of the maximal function
\[ \sup_{|\lambda| \ll N^{-s}} \left| \left( \sum_{\R_s} P(\xi - \lambda - a_j - a/q) \chi_s (\xi - \lambda - a_j - a/q) \times \hat{f}(\xi) \right)^{\vee} \right| \]
is bounded by $2^{\epsilon s}$ by Proposition \ref{0}. Replacing $\epsilon$ by $\epsilon/3$ completes the proof.
\end{proof}

\begin{lemma}\label{low}
For $1< p < 2$, the maximal function $\Ca_{\Lambda}^sf$ obeys
\[
\|\Ca_{\Lambda}^sf\|_{L^p(\RR) }\lesssim_\epsilon 2^{(\epsilon -2 + 4/p )s} \|f\|_{L^p(\RR)}
\]
for any $\epsilon > 0$.
\end{lemma}
\begin{proof}
By the triangle inequality, it is again enough to treat the maximal function with $|\lambda | \ll N^{-s}$ and all the frequencies shifted by some $a_j$.

We apply the previous lifting trick in the same context, with $\{\theta_1,\dots,\theta_K\}= a_j + \R_s$, and $c_k := \frac{\mu(q)}{\varphi(q)}$ if $\theta_k = a_j + a/q$. The operator norm of the $M_{\{c_j\},s}$ is $\lesssim 2^{(\epsilon-1)s}\cdot 2^{2s(\frac{1}{p} - \frac{1}{2})}$ by an application of Lemma \ref{K1} as in the proof of (\ref{vec-e30}); the operator norm of the maximal function
\[ \sup_{|\lambda| \ll N^{-s}} \left| \left( \sum_{\R_s} P(\xi - \lambda - a_j - a/q) \chi_s (\xi - \lambda - a_j - a/q) \times \hat{f}(\xi) \right)^{\vee} \right| \]
is bounded by $2^{2 s/p}$ by Proposition \ref{2}. Replacing $\epsilon$ by $\epsilon/2$ completes the proof.
\end{proof}

\begin{lemma}\label{high}
For $p > 2$, the maximal function $\Ca_{\Lambda}^sf$ obeys
\[
\|\Ca_{\Lambda}^sf\|_{L^{p}(\RR)} \lesssim_\epsilon 2^{(\epsilon + 1- 2/p)s} \|f\|_{L^{p}(\RR)}
\]
for any $\epsilon > 0$.
\end{lemma}
\begin{proof}
The argument is the same as in Lemma \ref{low}, except that the operator norm of $M_{\{c_j\},s}$ is $\lesssim 2^{(\epsilon-1)s}\cdot 2^{2s(\frac{1}{2} - \frac{1}{p})}$, and the operator norm of the maximal function is bounded by $2^{s}$.
\end{proof}

Interpolating between Lemma \ref{L2} and Lemma \ref{low}, with $\epsilon$ small enough depending on $p$, proves that there exists $\kappa_p > 0$ so that
\[ \| \Ca_{\Lambda}^sf \|_{L^p(\RR)} \lesssim 2^{- \kappa_p s} \|f\|_{L^p(\RR)} \]
provided $\frac{3}{2} < p \leq 2$. Simillarly, interpolating between Lemma \ref{L2} and Lemma \ref{high} we see that there exists $\kappa_p > 0$ so that
\[ \| \Ca_{\Lambda}^sf \|_{L^p(\RR)} \lesssim 2^{- \kappa_p s} \|f\|_{L^p(\RR)} \]
provided $2 \leq p < 4$.

Summing at last over $s \geq 0$ completes the proof of Theorem \ref{Main}.

\end{document}